\documentclass{amsart}

\usepackage{amssymb}
\usepackage{amsmath}
\usepackage{cite}
\usepackage{algorithmic}
\usepackage{bbm}
\usepackage{tikz}

\title[A construction of multiplicity class from Hesselink stratification]{A construction of multiplicity class of hypersurfaces from Hesselink stratification of a Hilbert scheme}
\author{Cheolgyu Lee}
\address[CL]{School of Mathematics, Korea Institute for Advanced Study(KIAS), 85 Hoegiro Dongdaemun-gu, Seoul, Republic of Korea}

\email{ghost279.math@gmail.com}
\thanks{2010 Mathematics Subject Classification. Primary 14L24. This work had been supported by a KIAS individual grant(6G067904) at Korea Institute for Advanced Study.}
\date{\today}

\newtheorem{theorem}{Theorem}[section]
\newtheorem{lemma}[theorem]{Lemma}

\newtheorem{corollary}[theorem]{Corollary}

\newenvironment{example}[1][Example]{\begin{trivlist}
\item[\hskip \labelsep {\bfseries #1}]}{\end{trivlist}}

\begin{document}
\begin{abstract}
It is well-known that there is a positive relationship between the maximal multiplicity and the length of the associated virtual 1-parameter subgroup of a projective hypersurface. In this paper, we will define the multiplicity classes of hypersurfaces and construct them from the Hesselink stratification of a Hilbert scheme.
\end{abstract}
\maketitle
\section{Introduction}
 In this paper, every scheme is over an algebraically closed field $k$. It is well-known that singularities of a semi-stable projective hypersurfaces are restricted \cite[p. 80]{GIT}. We can measure the magnitude of instability and singularity of a hypersurface and then compare them as in \cite{hesselsing}.  Consider a Hilbert scheme $\textup{Hilb}^{P_{r, d}}(\mathbb{P}_{k}^{r})$ of hypersurfaces \cite[p. 6]{Nitin} and Pl\"ucker embedding
\begin{equation}
\label{Plucker}
\textup{Hilb}^{P_{r, d}}(\mathbb{P}_{k}^{r})\xrightarrow{\sim} \mathbb{P}\left(\textup{H}^{0}\left(\mathbb{P}_{k}^{r}, \mathcal{O}_{\mathbb{P}_{k}^{r}}(d)\right)\right)
\end{equation} 
where 
\begin{displaymath}
P_{r, d}(x)=\binom{r+x}{r}-\binom{r+x-d}{r}
\end{displaymath}
for some $r, d\in\mathbb{N}$. There is the Hesselink stratification 
\begin{equation}
\label{stratification}
\textup{Hilb}^{P_{r, d}}(\mathbb{P}_{k}^{r})^{\textup{us}}=\coprod_{[\lambda], \delta}E_{[\lambda], \delta}^{d, r}
\end{equation}
 of the chosen Hilbert scheme $\textup{Hilb}^{P_{r, d}}(\mathbb{P}_{k}^{r})$ with respect to the canonical action of $\textup{SL}_{r+1}(k)$ and Pl\"ucker coordinate \eqref{Plucker}. In \cite[Theorem 3.1]{hesselsing}, it was shown that
\begin{equation}
\label{prevresult}
\frac{\Vert\lambda\Vert\delta-ad}{b-a}\leq \max_{p\in H_{q}}\textup{mult}_{p}H_{q}\leq\frac{rd}{r+1}-\delta\frac{a}{\Vert\lambda\Vert}
\end{equation}
if $H_{q}$ is represented by $q\in E_{[\lambda], \delta}^{d, r}$ for some 1-parameter subgroup $\lambda$ of $\textup{SL}_{r+1}(k)$ satisfying
\begin{displaymath}
\lambda(t)=\textup{diag}(t^{a_{0}}, t^{a_{1}}, \ldots, t^{a_{r}})\in\textup{SL}_{r+1}(k) \textrm{ for all }t\in k^{\times},
\end{displaymath}
$a=\min_{0\leq i\leq r}a_{i}$ and $b=\max_{0\leq i\leq r}a_{i}$ for some $\{a_{i}\}_{i=0}^{r}\in \mathbb{Z}^{r+1}$. Inequality \eqref{prevresult} determines the maximal multiplicity of hypersurface $H_{q}$ if the difference between two bounds in \eqref{prevresult} is less than $1$. Otherwise, \eqref{prevresult} cannot determine the maximal multiplicity. Also, \eqref{prevresult} cannot be used to distinguish the maximal multiplicities between two semi-stable hypersurfaces.

Is there a way to construct an arbitrary multiplicity class
\begin{displaymath}
S_{r, d, m}=\big\{q\in\textup{Hilb}^{P_{r, d}}(\mathbb{P}_{k}^{r})\big\vert \textup{mult}_{p} H_{q}=m\textrm{ for some }p\in H_{q}\big\}
\end{displaymath}
of hypersurfaces from Hesselink stratification in \eqref{stratification}? By \cite[Lemma 4.2]{hesselsing}, there is a coordinate $g\in\textup{SL}_{r+1}(k)$ such that both maximal multiplicity and instability of a hypersurface $H_{q}$ can be determined by the state polytope $\Delta_{g.H_{q}}$. Maximal multiplicity of $H_{q}$ is determined by a supporting hyperplane of $\Delta_{g.H_{q}}$ and instability of $H_{q}$ is determined by a sphere centered at the barycenter and tangent to $\Delta_{g.H_{q}}$. Comparing singularity and instability can be considered as comparing certain supporting hyperplanes and spheres, as we will see in figure \ref{example3}. If we increase the radius of the sphere, the sphere indicating instability looks like the hyperplane indicating maximal multiplicity around our state polytope.

To apply this simple idea, we will use the pull-back of the Hesselink stratification of $\textup{Hilb}^{P_{r, d+rN}}(\mathbb{P}_{k}^{r})^{\textup{us}}$ via the closed immersion
\begin{displaymath}
\phi_{r, d, N}:\textup{Hilb}^{P_{r, d}}(\mathbb{P}_{k}^{r})\rightarrow\textup{Hilb}^{P_{r, d+rN}}(\mathbb{P}_{k}^{r})
\end{displaymath}
given by a monomial multiplication. Such a map is {\it not} $\textup{SL}_{r+1}(k)$-equivariant, so we can map every point in $\textup{Hilb}^{P_{r, d}}(\mathbb{P}_{k}^{r})$ into the unstable locus of $\textup{Hilb}^{P_{r, d+rN}}(\mathbb{P}_{k}^{r})$ to deal with all hypersurfaces parametrized by $\textup{Hilb}^{P_{r, d}}(\mathbb{P}_{k}^{r})$. By taking $N\rightarrow \infty$, we can amplify the multiplicy of $H_{\phi_{r, d, N}(q)}$ at a fixed point $p\in H_{q}$, which will become a unique point of $H_{\phi_{r, d, N}(q)}$ attaining maximal multiplicity. Using \cite[Lemma 4.2]{hesselsing}, we can compare the Hesselink stratum containing $\phi_{r, d, N}(q)$ and the maximal multiplicity of $H_{\phi_{r, d, N}(q)}$, which can be directly computed from the multiplicity of $H_{q}$ at $p$.
In this paper (Corollary~\ref{last}), we will prove that $S_{r, d, m}$ can be constructed using the Hesselink stratification of another Hilbert scheme $\textup{Hilb}^{P_{r, d+rN}}(\mathbb{P}_{k}^{r})$ for sufficiently large $N$ using such an idea. 
\section{Preliminaries}
\subsection{Numerical criterion for semi-stability}
Suppose that $G=\textup{SL}_{r+1}(k)$ linearly acts on a vector space $V$. Then there is a $G$ action on $\mathbb{P}(V)$ satisfying $g.[v]=[g.v]$ for all $g\in G$ and $v\in V$. For any 1-parameter subgroup $\lambda\in\Gamma(G)$ of $G$, we have the weight decomposition \cite[Proposition 4.14]{Mukai}
\begin{displaymath}
V=\bigoplus_{m\in \mathbb{Z}}V_{m}
\end{displaymath}
where
\begin{displaymath}
V_{m}=\{v\in V| \lambda(t).v=t^{m}v \textup{ for all }t\in k^{\times}\}.
\end{displaymath}
Consequently, we may express an arbitrary $v\in V$ as the sum of eigenvectors $v=\sum_{m\in\mathbb{Z}}v_{m}$ where $v_{m}\in V_{m}$ for all $m\in\mathbb{Z}$. Let's define a function $\mu:\mathbb{P}(V)\times \Gamma(G)\rightarrow \mathbb{Z}$ as follows:
\begin{displaymath}
\mu([v], \lambda)=\min\{m\in\mathbb{Z}|v_{m}\neq 0\}.
\end{displaymath}
$[v]\in \mathbb{P}(V)$ is semi-stable if there is an invariant polynomial $f\in k[V]^{G}$ such that $f(v)\neq 0$. Now we are ready to state the numerical criterion for semi-stability \cite[Theorem 2.1 in p. 49]{GIT} in the above case.
\begin{theorem}
If $G$ linearly acts on $V$, then $x\in\mathbb{P}(V)$ is semi-stable if and only if
\begin{displaymath}
\mu(x,\lambda)\leq 0
\end{displaymath}
for all $\lambda\in\Gamma(G)$.
\end{theorem} 

\subsection{Hesselink stratification of a Hilbert scheme}
We can see that $G$ acts on $\Gamma(G)$ via conjugation. Let
\begin{displaymath}
(g\star\lambda) (t)=g\lambda(t)g^{-1}
\end{displaymath}
for all $g\in G$ and $t\in k^{\times}$. Note that $g\star\lambda\in \Gamma(G)$ so that $\star$ defines a $G$ action on $\Gamma(G)$. $\mu$ is invariant under the action of $G$ on $\mathbb{P}(V)\times\Gamma(G)$ \cite[p. 49]{GIT}; that is,
\begin{equation}
\mu(g.x, g\star\lambda)=\mu(x, \lambda)
\end{equation}
for all $g\in G$, $x\in\mathbb{P}(V)$ and $\lambda\in\Gamma(G)$. For an arbitrary $\lambda\in\Gamma(G)$ and $n\in\mathbb{N}$, we can define $n\lambda\in\Gamma(G)$ as follows:
\begin{displaymath}
(n\lambda)(t)=\lambda(t^{n}), \quad\forall t\in k^{\times}.
\end{displaymath}

Our $\mu(x, \lambda)$ measures how much $x\in\mathbb{P}(V)$ is unstable under the action of the image of $\lambda$. To measure the magnitude of instability of $x\in\mathbb{P}(V)$ under the action of $G$, we may normalize $\mu$ by some norm $\Vert\cdot\Vert:\Gamma(G)\rightarrow\mathbb{R}_{\geq 0}$ satisfying
\begin{itemize}
\item $\Vert n\lambda\Vert=n\Vert \lambda\Vert$ for all $n\in\mathbb{N}$ and $\lambda\in\Gamma(G)$,
\item $\Vert g\star\lambda\Vert=\Vert\lambda\Vert$ for all $g\in G$ and $\lambda\in\Gamma(G)$ 
\end{itemize}
because $\mu(x, n\lambda)=n\mu(x, \lambda)$ for all $n\in\mathbb{N}$ and $\lambda\in\Gamma(G)$. Furthermore, we may consider the value
\begin{equation}
\label{kempfindex}
\max_{\lambda\in\Gamma(G)}\frac{\mu(x, \lambda)}{\Vert\lambda\Vert}
\end{equation}
as the magnitude of the instability of $x\in\mathbb{P}(V)$. Let us state a theorem on the existence of \eqref{kempfindex}, which had been proven in \cite{Kempf}.
\begin{theorem}[Kempf, \cite{Kempf}]
\label{Kempftheorem}
Suppose that the norm $\Vert\cdot\Vert$ satisfies the above conditions and there is a maximal torus $T$ of $G$ and integral-valued bilinear form $\langle\cdot, \cdot\rangle$ on the lattice $\Gamma(T)$ such that $\langle\lambda, \lambda\rangle=\Vert\lambda\Vert^{2}$ for all $\lambda\in\Gamma(T)$ . If $x\in\mathbb{P}(V)$ is an unstable point, then there is $\sigma\in\Gamma(G)$ and a parabolic subgroup $G_{x}$ of $G$ satisfying
\begin{displaymath}
\frac{\mu(x, \sigma)}{\Vert\sigma\Vert}=\max_{\lambda\in\Gamma(G)}\frac{\mu(x, \lambda)}{\Vert\lambda\Vert}
\end{displaymath}
and
\begin{displaymath}
G_{x}=\pi(\sigma)=\big\lbrace g\in G\big\vert \exists\lim_{t\rightarrow 0} \sigma(t)g\sigma(t^{-1})\in G \big\rbrace
\end{displaymath}
such that $G_{x}$ acts transitively on the set
\begin{displaymath}
\Lambda_{x}=\bigg\lbrace\rho\in\Gamma(G)\bigg \vert \frac{\mu(x, \rho)}{\Vert\rho\Vert}=\max_{\lambda\in\Gamma(G)}\frac{\mu(x, \lambda)}{\Vert\lambda\Vert} \bigg\rbrace
\end{displaymath}
via the conjugation action given by $\star$. Furthermore, $G_{x}=\pi(\rho)$ for all $\rho\in \Lambda_{x}$.
\end{theorem}
Let
\begin{displaymath}
E_{[\lambda], \delta}=\bigg\lbrace x\in\mathbb{P}(V)\bigg\vert\max_{\lambda\in\Gamma(G)}\frac{\mu(x, \lambda)}{\Vert\lambda\Vert}=\delta,\quad \Lambda_{x}\cap [\lambda]\neq\emptyset \bigg\rbrace
\end{displaymath}
where $[\lambda]$ is the conjugacy class of $\Gamma(G)$ containing $\lambda\in\Gamma(G)$ and $\delta\in\mathbb{R}_{>0}$. Each $E_{[\lambda], \delta}$ is a locally closed subset of $\mathbb{P}(V)$ as we can see in \cite{Hesselink}. 
From now on, let $V$ be $k[x_{0}, \ldots, x_{r}]_{d}=\textup{Sym}^{d}k[x_{0}, \ldots, x_{r}]_{1}$ and let the action be the $G$-action obtained by symmetrizing the $G$-action on $k[x_{0}, \ldots, x_{r}]_{1}$ given by the formula
\begin{displaymath}
g.x_{i}=\sum_{j=0}^{r}g_{ji}x_{j}, \textrm{ for all }g\in G\textrm{ and }0\leq i\leq r.
\end{displaymath}
We see that $\mathbb{P}(V)$ is isomorphic to the Hilbert scheme of hypersurfaces of dimension $r-1$ and degree $d$.

Let $T$ be the group of diagonal matrices in $G$. $T$ is a maximal torus of $G$. Let $X(T)$ be the group of characters defined on $T$ and 
\begin{displaymath}
\chi_{i}(t)=t_{ii}
\end{displaymath}
for all $t\in T$ and $i\in\{0, 1, \ldots, r\}$, where $t_{ii}$ is the $i$'th diagonal entry of $t$. We may embed $X(T)\otimes_{\mathbb{Z}}\mathbb{R}$ into $\mathbb{R}^{r+1}$ via map $\iota: X(T)\otimes_{\mathbb{Z}}\mathbb{R}\rightarrow \mathbb{R}^{r+1}$ satisfying 
\begin{displaymath}
\iota(\chi_{i})=\mathbf{e}_{i}-\frac{1}{r+1}\sum_{j=0}^{r}\mathbf{e}_{j},\quad\forall i\in\{0, 1, \ldots, r\}
\end{displaymath}
where $\mathbf{e}_{i}$ is the $i$'th elementary vector in $\mathbb{R}^{r+1}$. There is a perfect pairing $\langle\cdot, \cdot\rangle: X(T)\otimes_{\mathbb{Z}} \Gamma(T)\rightarrow\mathbb{Z}$ satisfying
\begin{displaymath}
\chi(\lambda(t))=t^{\langle\chi,\lambda\rangle}
\end{displaymath}
for all $\chi\in X(T)$, $\lambda\in\Gamma(T)$ and $t\in k^{\times}$. Consider the basis $\{\lambda_{i}\}_{i=0}^{r}$ of $\Gamma(T)$, which is dual to $\{\chi_{i}\}_{i=0}^{r}$ with respect to the pairing  $\langle\cdot, \cdot\rangle$. Considering the isomorphism $h:\Gamma(T)\rightarrow X(T)$ satisfying $h(\lambda_{i})=\chi_{i}$ for all $0\leq i\leq r$, we can define a norm $\Vert\cdot\Vert_{0}$ on $\Gamma(T)$ induced by the Euclidean norm $\vert\cdot\vert$ of $\mathbb{R}^{r+1}$. Such a norm is invariant under the conjugation action of the Weyl group of $T$ in $G$. Therefore, we can extend $\Vert\cdot\Vert_{0}$ to the norm $\Vert\cdot\Vert$ of $\Gamma(G)$ via conjugation. That is, for an arbitrary $\lambda\in\Gamma(G)$,
\begin{displaymath}
\Vert\lambda\Vert=\Vert g\star\lambda\Vert_{0}
\end{displaymath}
if $g\star\lambda\in\Gamma(T)$ for some $g\in G$. Such a $g$ always exists since all maximal tori of $G$ are conjugate. Note that $\Vert\cdot\Vert$ satisfies all the hypotheses in Thoerem~\ref{Kempftheorem}.

 Consequently, we have the Hesselink stratification of the Hilbert scheme
\begin{displaymath}
\textup{Hilb}^{P_{r, d}}(\mathbb{P}_{k}^{r})^{\textup{us}}=\coprod_{[\lambda], \delta}E_{[\lambda], \delta}^{d, r}
\end{displaymath}
where 
\begin{displaymath}
P_{r, d}(x)=\binom{r+x}{r}-\binom{r+x-d}{r}
\end{displaymath}
for each $d$ and $r$ in $\mathbb{N}$.
\subsection{State polytope and instability}
Let the state polytope\footnote[1]{According to the literature, such a definition of state polytope is obtained from the canonical $\textup{GL}_{r+1}(k)$-action with the canonical linearization twisted by a power of determinant while we are considering the canonical $\textup{SL}_{r+1}(k)$-action. However, such a set-up does not changes the situation of our problem in the viewpoint of GIT as we can see in \cite[2.2]{instability}. Also, such a definition of state polytope let us observe the symmetry within our problem directly from the picture.} $\Delta_{x}$ of a Hilbert point $x\in \textup{Hilb}^{P_{r, d}}(\mathbb{P}_{k}^{r})$ be the newton polytope of the defining equation of $x$, which is embedded in $\mathbb{R}^{r+1}$. Let $\vert\Delta_{x}\vert$ be the Euclidean distance from $\xi_{r, d}=\frac{d}{r+1}\mathbbm{1}$ to $\Delta_{x}$ where $\mathbbm{1}$ is the all-$1$ vector in $\mathbb{R}^{r+1}$. There is a unique 1-parameter subgroup $\lambda_{x}\in\Gamma(T)$ which is indivisible by any non-unit integer and
\begin{displaymath}
\iota \left(h(\lambda_{x})\otimes_{\mathbb{Z}} 1\right)
\end{displaymath}
is the distance vector from $\xi_{r, d}=\frac{d}{r+1}\mathbbm{1}$ to $\Delta_{x}$. It is well-known that we can measure the instability of an arbitrary $x\in \textup{Hilb}^{P_{r, d}}(\mathbb{P}_{k}^{r})$ using some conditions on state polytopes.
\begin{theorem}
\label{polytope}
$p\in \textup{Hilb}^{P_{r, d}}(\mathbb{P}_{k}^{r})$ is semi-stable if and only if $\xi_{r, d}\in\Delta_{g.p}$ for all $g\in G$. Furthermore, if $p\in \textup{Hilb}^{P_{r, d}}(\mathbb{P}_{k}^{r})$ is unstable, then there is $g_{m}\in G$ satisfying
\begin{displaymath}
\vert\Delta_{g_{m}.p}\vert = \max_{g\in G}\vert\Delta_{g.p}\vert
\end{displaymath}
and $p\in E_{[\lambda], \delta}^{d, r}$ if and only if $\delta=\vert\Delta_{g_{m}.p}\vert$ and $\lambda_{g_{m}.p} \in [\lambda]$.
\end{theorem}
\begin{proof}
See \cite{Ian} for semi-stability case. See \cite{Hesselink} and \cite{Kempf} for the remainder. This theorem also can be derived from \cite[Theorem 2.2]{instability}.
\end{proof}
From now on, we will call $\Delta_{g_{m}.x}$ in Theorem~\ref{polytope} as a worst state polytope of $x$.
\begin{example}
Let $k=\mathbb{C}$ and $f=y^{3}+2xy^{2}+yz^{2}\in \mathbb{C}[x, y, z]_{3}$ and let $[f]$ be the class in $\mathbb{P}(\mathbb{C}[x, y, z]_{3})$ containing $f$. The state polytope $\Delta_{[f]}$ looks like the gray triangle in figure \ref{example1}. 
\begin{figure}[h!]
\begin{tikzpicture}
\coordinate (up) at (0,sqrt 3);
\coordinate (left) at (-1.5,-1/2*sqrt 3);
\coordinate (right) at (1.5,-1/2*sqrt 3);
\filldraw [black] (barycentric cs:up=1,left=1,right=1) circle (1pt);
\node at (barycentric cs:up=0.7,left=1,right=1) {$\xi_{2, 3}$};
\filldraw [black] (barycentric cs:up=1,left=0,right=0) circle (1pt);
\node at (barycentric cs:up=1.1,left=-0.1,right=0) {$y^{3}$};
\filldraw [black] (barycentric cs:up=2,left=1,right=0) circle (1pt);
\node at (barycentric cs:up=2,left=1.4,right=-0.4) {$xy^{2}$};
\filldraw [black] (barycentric cs:up=2,left=0,right=1) circle (1pt);
\filldraw [black] (barycentric cs:up=1,left=2,right=0) circle (1pt);
\filldraw [black] (barycentric cs:up=1,left=0,right=2) circle (1pt);
\node at (barycentric cs:up=1,left=-0.4,right=2.4) {$yz^{2}$};
\filldraw [black] (barycentric cs:up=0,left=2,right=1) circle (1pt);
\filldraw [black] (barycentric cs:up=0,left=1,right=2) circle (1pt);
\filldraw [black] (barycentric cs:up=0,left=1,right=0) circle (1pt);
\filldraw [black] (barycentric cs:up=0,left=0,right=1) circle (1pt);
\filldraw [gray] (barycentric cs:up=1,left=0,right=0)--(barycentric cs:up=2,left=1,right=0)--(barycentric cs:up=1,left=0,right=2);
\end{tikzpicture}
\caption{The state polytope of $[y^{3}+2xy^{2}+yz^{2}]$.}
\label{example1}
\end{figure}\\
It is easy to observe that $f$ is a product of two polynomials $y$ and $y^{2}+2xy+z^{2}$ , which are both irreducible. We can also notice that $g\in \mathbb{C}[x, y, z]_{3}$ must be the product of some linear polynomials if $g$ satisfies $|\Delta_{[g]}|>|\Delta_{[f]}|$. That is, $ \vert\Delta_{[f]}\vert = \max_{g\in \textup{SL}_{r+1}(\mathbb{C})}\vert\Delta_{g.[f]}\vert$, so that $[f]\in E^{3, 2}_{[\lambda], \sqrt{2}/2}$ where $\lambda(t)=\textup{diag} (t, t^{-1}, 1)\in \textup{SL}_{r+1}(\mathbb{C})$ for all $t\in\mathbb{C}^{\times}$. We will use this example throughout the next section.
\end{example}

\section{Multiplicities and adapted-1PS of a destabilization}
We can see that the multiplicity of a hypersurface $H_{p}$ represented by $p\in \textup{Hilb}^{P_{r, d}}(\mathbb{P}_{k}^{r})$ at a point $[1:0:, \ldots, : 0]\in\mathbb{P}_{k}^{r}$ is determined by a supporting hyperplane of $\Delta_{p}$ in \cite[Lemma 4.1]{hesselsing}. For example, if $p=[y^{3}+2xy^{2}+yz^{2}]\in \mathbb{P}\left(\mathbb{C}[x, y, z]_{3}\right)\cong\textup{Hilb}^{P_{2, 3}}\left(\mathbb{P}^{2}_{\mathbb{C}}\right)$, then we can compute the multiplicity of the curve represented by $x$ at $[1:0:0]$ by looking at figure \ref{example2}.\\
\begin{figure}[h!]
\begin{tikzpicture}
\coordinate (up) at (0,sqrt 3);
\coordinate (left) at (-1.5,-1/2*sqrt 3);
\coordinate (right) at (1.5,-1/2*sqrt 3);
\filldraw [black] (barycentric cs:up=1,left=1,right=1) circle (1pt);
\node at (barycentric cs:up=0.7,left=1.3,right=0.7) {$\xi_{2, 3}$};
\filldraw [black] (barycentric cs:up=1,left=0,right=0) circle (1pt);
\node at (barycentric cs:up=1.1,left=-0.1,right=0) {$y^{3}$};
\filldraw [black] (barycentric cs:up=2,left=1,right=0) circle (1pt);
\node at (barycentric cs:up=2,left=1.4,right=-0.4) {$xy^{2}$};
\filldraw [black] (barycentric cs:up=2,left=0,right=1) circle (1pt);
\filldraw [black] (barycentric cs:up=1,left=2,right=0) circle (1pt);
\filldraw [black] (barycentric cs:up=1,left=0,right=2) circle (1pt);
\node at (barycentric cs:up=1,left=-0.4,right=2.4) {$yz^{2}$};
\filldraw [black] (barycentric cs:up=0,left=2,right=1) circle (1pt);
\filldraw [black] (barycentric cs:up=0,left=1,right=2) circle (1pt);
\filldraw [black] (barycentric cs:up=0,left=1,right=0) circle (1pt);
\filldraw [black] (barycentric cs:up=0,left=0,right=1) circle (1pt);
\filldraw [gray] (barycentric cs:up=1,left=0,right=0)--(barycentric cs:up=2,left=1,right=0)--(barycentric cs:up=1,left=0,right=2);
\draw [black,dashed] (barycentric cs:up=1.5,left=2,right=-0.5)--(barycentric cs:up=-0.5,left=2,right=1.5);
\node at (barycentric cs:up=-0.7,left=1.7,right=1.8) {$1$};
\draw [black,dashed] (barycentric cs:up=0.5,left=3,right=-0.5)--(barycentric cs:up=-0.5,left=3,right=0.5);
\node at (barycentric cs:up=-0.7,left=3.7,right=-0.2) {multiplicity $0$};
\draw [black] (barycentric cs:up=2.5,left=1,right=-0.5)--(barycentric cs:up=-0.5,left=1,right=2.5);
\node at (barycentric cs:up=-0.7,left=0.7,right=2.8) {$2$};
\end{tikzpicture}
\caption{The lines which determine the multiplicity at $[1:0:0]$.}
\label{example2}
\end{figure} 
Theorem~\ref{polytope} means that the Kempf index is the radius of the sphere which is centered at $\xi_{r, d}$ and tangent to a worst state polytope, as we can see in figure \ref{example1}. Therefore, comparing instability and singularity of a hypersurface can be considered as comparing hyperplane and sphere. When does a sphere looks like a hyperplane? We may increase the radius of the sphere and look at it locally around state polytope, by multiplying a monomial. For example, the state polytope $\Delta_{[y^{N}z^{N}(y^{3}+2xy^{2}+yz^{2})]}$ of $[y^{N}z^{N}(y^{3}+2xy^{2}+yz^{2})]\in \left(\mathbb{C}[x, y, z]_{3+2N}\right)\cong\textup{Hilb}^{P_{2, 3+2N}}\left(\mathbb{P}^{2}_{\mathbb{C}}\right)$ looks like the gray triangle in figure \ref{example3} when $N\in \{0, 2, 4\}$. We can see that sphere approaches hyperplane around $\Delta_{[y^{N}z^{N}(y^{3}+2xy^{2}+yz^{2})]}$ as $N\rightarrow\infty$ in figure \ref{example3}.
\begin{figure}[h!]
\hfill
\begin{tikzpicture}
\coordinate (up) at (0,sqrt 3);
\coordinate (left) at (-1.5,-1/2*sqrt 3);
\coordinate (right) at (1.5,-1/2*sqrt 3);
\filldraw [black] (barycentric cs:up=1,left=1,right=1) circle (1pt);
\filldraw [black] (barycentric cs:up=1,left=0,right=0) circle (1pt);
\node at (barycentric cs:up=1.1,left=-0.1,right=0) {$y^{3}$};
\filldraw [black] (barycentric cs:up=2,left=1,right=0) circle (1pt);
\node at (barycentric cs:up=0,left=1.1,right=-0.1) {$x^{3}$};
\filldraw [black] (barycentric cs:up=2,left=0,right=1) circle (1pt);
\filldraw [black] (barycentric cs:up=1,left=2,right=0) circle (1pt);
\filldraw [black] (barycentric cs:up=1,left=0,right=2) circle (1pt);
\node at (barycentric cs:up=0,left=-0.15,right=1.15) {$z^{3}$};
\filldraw [black] (barycentric cs:up=0,left=2,right=1) circle (1pt);
\filldraw [black] (barycentric cs:up=0,left=1,right=2) circle (1pt);
\filldraw [black] (barycentric cs:up=0,left=1,right=0) circle (1pt);
\filldraw [black] (barycentric cs:up=0,left=0,right=1) circle (1pt);
\filldraw [gray] (barycentric cs:up=1,left=0,right=0)--(barycentric cs:up=2,left=1,right=0)--(barycentric cs:up=1,left=0,right=2);
\draw (barycentric cs:up=1,left=1,right=1) circle [radius=.5];
\draw (barycentric cs:up=2.5,left=1,right=-0.5)--(barycentric cs:up=-0.5,left=1,right=2.5);
\node at (barycentric cs:up=2,left=2,right=-1) {$N=0$};
\end{tikzpicture}
\hfill
\begin{tikzpicture}
\coordinate (up) at (0,sqrt 3);
\coordinate (left) at (-1.5,-1/2*sqrt 3);
\coordinate (right) at (1.5,-1/2*sqrt 3);
\node at (barycentric cs:up=1.1,left=-0.1,right=0) {$y^{7}$};
\node at (barycentric cs:up=0,left=1.1,right=-0.1) {$x^{7}$};
\node at (barycentric cs:up=0,left=-0.15,right=1.15) {$z^{7}$};
\filldraw [black] (barycentric cs:up=1,left=0,right=0) circle (1pt);
\filldraw [black] (barycentric cs:up=6,left=1,right=0) circle (1pt);
\filldraw [black] (barycentric cs:up=6,left=0,right=1) circle (1pt);
\filldraw [black] (barycentric cs:up=5,left=2,right=0) circle (1pt);
\filldraw [black] (barycentric cs:up=5,left=1,right=1) circle (1pt);
\filldraw [black] (barycentric cs:up=5,left=0,right=2) circle (1pt);
\filldraw [black] (barycentric cs:up=4,left=3,right=0) circle (1pt);
\filldraw [black] (barycentric cs:up=4,left=2,right=1) circle (1pt);
\filldraw [black] (barycentric cs:up=4,left=1,right=2) circle (1pt);
\filldraw [black] (barycentric cs:up=4,left=0,right=3) circle (1pt);
\filldraw [black] (barycentric cs:up=3,left=4,right=0) circle (1pt);
\filldraw [black] (barycentric cs:up=3,left=3,right=1) circle (1pt);
\filldraw [black] (barycentric cs:up=3,left=2,right=2) circle (1pt);
\filldraw [black] (barycentric cs:up=3,left=1,right=3) circle (1pt);
\filldraw [black] (barycentric cs:up=3,left=0,right=4) circle (1pt);
\filldraw [black] (barycentric cs:up=2,left=5,right=0) circle (1pt);
\filldraw [black] (barycentric cs:up=2,left=4,right=1) circle (1pt);
\filldraw [black] (barycentric cs:up=2,left=3,right=2) circle (1pt);
\filldraw [black] (barycentric cs:up=2,left=2,right=3) circle (1pt);
\filldraw [black] (barycentric cs:up=2,left=1,right=4) circle (1pt);
\filldraw [black] (barycentric cs:up=2,left=0,right=5) circle (1pt);
\filldraw [black] (barycentric cs:up=1,left=6,right=0) circle (1pt);
\filldraw [black] (barycentric cs:up=1,left=5,right=1) circle (1pt);
\filldraw [black] (barycentric cs:up=1,left=4,right=2) circle (1pt);
\filldraw [black] (barycentric cs:up=1,left=3,right=3) circle (1pt);
\filldraw [black] (barycentric cs:up=1,left=2,right=4) circle (1pt);
\filldraw [black] (barycentric cs:up=1,left=1,right=5) circle (1pt);
\filldraw [black] (barycentric cs:up=1,left=0,right=6) circle (1pt);
\filldraw [black] (barycentric cs:up=0,left=7,right=0) circle (1pt);
\filldraw [black] (barycentric cs:up=0,left=6,right=1) circle (1pt);
\filldraw [black] (barycentric cs:up=0,left=5,right=2) circle (1pt);
\filldraw [black] (barycentric cs:up=0,left=4,right=3) circle (1pt);
\filldraw [black] (barycentric cs:up=0,left=3,right=4) circle (1pt);
\filldraw [black] (barycentric cs:up=0,left=2,right=5) circle (1pt);
\filldraw [black] (barycentric cs:up=0,left=1,right=6) circle (1pt);
\filldraw [black] (barycentric cs:up=0,left=0,right=7) circle (1pt);
\filldraw [gray] (barycentric cs:up=5,left=0,right=2)--(barycentric cs:up=4,left=1,right=2)--(barycentric cs:up=3,left=0,right=4);
\draw (barycentric cs:up=1,left=1,right=1) circle [radius=9/14];
\draw (barycentric cs:up=6.5,left=1,right=-0.5)--(barycentric cs:up=-0.5,left=1,right=6.5);
\node at (barycentric cs:up=2,left=2,right=-1) {$N=2$};
\end{tikzpicture}
\hfill
\begin{tikzpicture}
\coordinate (up) at (0,sqrt 3);
\coordinate (left) at (-1.5,-1/2*sqrt 3);
\coordinate (right) at (1.5,-1/2*sqrt 3);
\filldraw [black] (barycentric cs:up=11,left=0,right=0) circle (1pt);
\node at (barycentric cs:up=1.1,left=-0.1,right=0) {$y^{11}$};
\filldraw [black] (barycentric cs:up=10,left=1,right=0) circle (1pt);
\filldraw [black] (barycentric cs:up=10,left=0,right=1) circle (1pt);
\filldraw [black] (barycentric cs:up=9,left=2,right=0) circle (1pt);
\filldraw [black] (barycentric cs:up=9,left=1,right=1) circle (1pt);
\filldraw [black] (barycentric cs:up=9,left=0,right=2) circle (1pt);
\filldraw [black] (barycentric cs:up=8,left=3,right=0) circle (1pt);
\filldraw [black] (barycentric cs:up=8,left=2,right=1) circle (1pt);
\filldraw [black] (barycentric cs:up=8,left=1,right=2) circle (1pt);
\filldraw [black] (barycentric cs:up=8,left=0,right=3) circle (1pt);
\filldraw [black] (barycentric cs:up=7,left=4,right=0) circle (1pt);
\filldraw [black] (barycentric cs:up=7,left=3,right=1) circle (1pt);
\filldraw [black] (barycentric cs:up=7,left=2,right=2) circle (1pt);
\filldraw [black] (barycentric cs:up=7,left=1,right=3) circle (1pt);
\filldraw [black] (barycentric cs:up=7,left=0,right=4) circle (1pt);
\filldraw [black] (barycentric cs:up=6,left=5,right=0) circle (1pt);
\filldraw [black] (barycentric cs:up=6,left=4,right=1) circle (1pt);
\filldraw [black] (barycentric cs:up=6,left=3,right=2) circle (1pt);
\filldraw [black] (barycentric cs:up=6,left=2,right=3) circle (1pt);
\filldraw [black] (barycentric cs:up=6,left=1,right=4) circle (1pt);
\filldraw [black] (barycentric cs:up=6,left=0,right=5) circle (1pt);
\filldraw [black] (barycentric cs:up=5,left=6,right=0) circle (1pt);
\filldraw [black] (barycentric cs:up=5,left=5,right=1) circle (1pt);
\filldraw [black] (barycentric cs:up=5,left=4,right=2) circle (1pt);
\filldraw [black] (barycentric cs:up=5,left=3,right=3) circle (1pt);
\filldraw [black] (barycentric cs:up=5,left=2,right=4) circle (1pt);
\filldraw [black] (barycentric cs:up=5,left=1,right=5) circle (1pt);
\filldraw [black] (barycentric cs:up=5,left=0,right=6) circle (1pt);
\filldraw [black] (barycentric cs:up=4,left=7,right=0) circle (1pt);
\filldraw [black] (barycentric cs:up=4,left=6,right=1) circle (1pt);
\filldraw [black] (barycentric cs:up=4,left=5,right=2) circle (1pt);
\filldraw [black] (barycentric cs:up=4,left=4,right=3) circle (1pt);
\filldraw [black] (barycentric cs:up=4,left=3,right=4) circle (1pt);
\filldraw [black] (barycentric cs:up=4,left=2,right=5) circle (1pt);
\filldraw [black] (barycentric cs:up=4,left=1,right=6) circle (1pt);
\filldraw [black] (barycentric cs:up=4,left=0,right=7) circle (1pt);
\filldraw [black] (barycentric cs:up=3,left=8,right=0) circle (1pt);
\filldraw [black] (barycentric cs:up=3,left=7,right=1) circle (1pt);
\filldraw [black] (barycentric cs:up=3,left=6,right=2) circle (1pt);
\filldraw [black] (barycentric cs:up=3,left=5,right=3) circle (1pt);
\filldraw [black] (barycentric cs:up=3,left=4,right=4) circle (1pt);
\filldraw [black] (barycentric cs:up=3,left=3,right=5) circle (1pt);
\filldraw [black] (barycentric cs:up=3,left=2,right=6) circle (1pt);
\filldraw [black] (barycentric cs:up=3,left=1,right=7) circle (1pt);
\filldraw [black] (barycentric cs:up=3,left=0,right=8) circle (1pt);
\filldraw [black] (barycentric cs:up=2,left=9,right=0) circle (1pt);
\filldraw [black] (barycentric cs:up=2,left=8,right=1) circle (1pt);
\filldraw [black] (barycentric cs:up=2,left=7,right=2) circle (1pt);
\filldraw [black] (barycentric cs:up=2,left=6,right=3) circle (1pt);
\filldraw [black] (barycentric cs:up=2,left=5,right=4) circle (1pt);
\filldraw [black] (barycentric cs:up=2,left=4,right=5) circle (1pt);
\filldraw [black] (barycentric cs:up=2,left=3,right=6) circle (1pt);
\filldraw [black] (barycentric cs:up=2,left=2,right=7) circle (1pt);
\filldraw [black] (barycentric cs:up=2,left=1,right=8) circle (1pt);
\filldraw [black] (barycentric cs:up=2,left=0,right=9) circle (1pt);
\filldraw [black] (barycentric cs:up=1,left=10,right=0) circle (1pt);
\filldraw [black] (barycentric cs:up=1,left=9,right=1) circle (1pt);
\filldraw [black] (barycentric cs:up=1,left=8,right=2) circle (1pt);
\filldraw [black] (barycentric cs:up=1,left=7,right=3) circle (1pt);
\filldraw [black] (barycentric cs:up=1,left=6,right=4) circle (1pt);
\filldraw [black] (barycentric cs:up=1,left=5,right=5) circle (1pt);
\filldraw [black] (barycentric cs:up=1,left=4,right=6) circle (1pt);
\filldraw [black] (barycentric cs:up=1,left=3,right=7) circle (1pt);
\filldraw [black] (barycentric cs:up=1,left=2,right=8) circle (1pt);
\filldraw [black] (barycentric cs:up=1,left=1,right=9) circle (1pt);
\filldraw [black] (barycentric cs:up=1,left=0,right=10) circle (1pt);
\filldraw [black] (barycentric cs:up=0,left=11,right=0) circle (1pt);
\node at (barycentric cs:up=0,left=1.1,right=-0.1) {$x^{11}$};
\filldraw [black] (barycentric cs:up=0,left=10,right=1) circle (1pt);
\filldraw [black] (barycentric cs:up=0,left=9,right=2) circle (1pt);
\filldraw [black] (barycentric cs:up=0,left=8,right=3) circle (1pt);
\filldraw [black] (barycentric cs:up=0,left=7,right=4) circle (1pt);
\filldraw [black] (barycentric cs:up=0,left=6,right=5) circle (1pt);
\filldraw [black] (barycentric cs:up=0,left=5,right=6) circle (1pt);
\filldraw [black] (barycentric cs:up=0,left=4,right=7) circle (1pt);
\filldraw [black] (barycentric cs:up=0,left=3,right=8) circle (1pt);
\filldraw [black] (barycentric cs:up=0,left=2,right=9) circle (1pt);
\filldraw [black] (barycentric cs:up=0,left=1,right=10) circle (1pt);
\filldraw [black] (barycentric cs:up=0,left=0,right=11) circle (1pt);
\node at (barycentric cs:up=0,left=-0.15,right=1.15) {$z^{11}$};
\filldraw [gray] (barycentric cs:up=7,left=0,right=4)--(barycentric cs:up=6,left=1,right=4)--(barycentric cs:up=5,left=0,right=6);
\draw (barycentric cs:up=1,left=1,right=1) circle [radius=sqrt{57}/11];
\draw (barycentric cs:up=11,left=1,right=-1)--(barycentric cs:up=-1,left=1,right=11);
\node at (barycentric cs:up=2,left=2,right=-1) {$N=4$};
\end{tikzpicture}
\caption{Supporting hyperplane and circle of the state polytope of $[y^{N}z^{N}(y^{3}+2xy^{2}+yz^{2})]$, when $N$ is equal to  $0$, $2$ and $4$.}
\label{example3}
\end{figure}
To state this idea generally, we need to embed the Hilbert scheme via morphism
\begin{displaymath}
\phi_{r, d, N}:\textup{Hilb}^{P_{r, d}}(\mathbb{P}_{k}^{r})\rightarrow\textup{Hilb}^{P_{r, d+rN}}(\mathbb{P}_{k}^{r})
\end{displaymath}
which maps $[f]\in \mathbb{P}(k[x_{0}, x_{1}, \ldots, x_{r}]_{d})\cong\textup{Hilb}^{P_{d}}(\mathbb{P}_{k}^{r})$ to
\begin{displaymath}
\left[ f\prod_{i=1}^{r}x_{i}^{N}\right] \in \mathbb{P}(k[x_{0}, x_{1}, \ldots, x_{r}]_{d+rN})\cong\textup{Hilb}^{P_{d+rN}}(\mathbb{P}_{k}^{r}).
\end{displaymath}
For an arbitrary $q\in\textup{Hilb}^{P_{r, d}}(\mathbb{P}_{k}^{r})$, $\Delta_{\phi_{r, d, N}(q)}$ is contained in the polytope
\begin{displaymath}
Q_{d, N}^{r}=\bigg\{(y_{0}, y_{1}, \ldots, y_{r})\in\mathbb{R}^{r+1}\bigg\vert \sum_{i=0}^{r}y_{i}=d+rN, y_{i}\geq N\textrm{ for all }1\leq i\leq r, y_{0}\geq 0 \bigg\}.
\end{displaymath}
$\Delta_{\phi_{r, d, N}(q)}$ is not always a worst state polytope of $\phi_{r, d, N}(q)$. However, the distance vector from $\xi_{r, d+rN}$ to a worst state polytope of $\phi_{r, d, N}(q)$ is lying in some bounded region if $N> d$. Let $l_{d, N, m}^{r}$ be the Euclidean distance from $\xi_{r, d+rN}$ to the point $(d-m, m+N, N, \ldots, N)\in\mathbb{R}^{r+1}$. We can check that
\begin{equation}
\label{maxlength}
l_{d, N, m}^{r}=\max\bigg\{\vert\xi_{r, d+rN}-y\vert\bigg\vert y=(y_{i})_{i=0}^{r}\in Q_{d, N}^{r}, y_{0}=d-m\bigg\}.
\end{equation}
Let
\begin{displaymath}
B_{d, N, m}^{r}=\bigg\{y=(y_{i})_{i=0}^{r}\in\mathbb{R}^{r+1}_{\geq 0}\bigg\vert \sum_{i=0}^{r}y_{i}=d+rN, \vert\xi_{r, d+rN}-y \vert\leq l^{r}_{d, N, m}, y_{0}\leq d-m \bigg\}.
\end{displaymath}
The following lemma shows that we can derive a condition satisfied by the associated 1-parameter subgroup and the Kempf index of $\phi_{r, d, N}(q)$, if we know $\textup{mult}_{[1:0:\ldots:0]}H_{q}$.
\begin{lemma}
\label{restriction}
Suppose that $q\in\textup{Hilb}^{P_{r, d}}(\mathbb{P}_{k}^{r})$, $\textup{mult}_{[1:0:\ldots:0]}H_{q}=m$ and $N> d$. Then, $\phi_{r, d, N}(q)\in E_{[\lambda], \delta}^{d+Nr, r}$ for some $\lambda\in\Gamma(T)$ and $\delta>0$ satisfying
\begin{displaymath}
\frac{\delta}{\Vert\lambda\Vert}\iota(h(\lambda)\otimes_{\mathbb{Z}}1)\in B_{d, N, m}^{r}.
\end{displaymath}
\end{lemma}
\begin{proof}
Let $H^{N}$ be the hypersurface embedded in $\mathbb{P}_{k}^{r}$ defined by the polynomial $\prod_{i=1}^{r}x_{i}^{N}$. $H^{N}$ has a unique closed point $[1:0:\ldots:0]$ of maximal multiplicity and $\textup{mult}_{p}H^{N}\leq (r-1)N$ for $p\neq[1:0:\ldots:0]
$. 

 Note that
\begin{displaymath}
\textup{mult}_{p}H_{\phi_{r, d, N}(q)}=\textup{mult}_{p}H_{q}+\textup{mult}_{p}H^{N}.
\end{displaymath}
for all $p\in \mathbb{P}_{k}^{r}$. If $p=[1:0:\ldots:0]$, then $\textup{mult}_{p}H_{\phi_{r, d, N}(q)}=m+rN$. Otherwise,
\begin{displaymath}
\textup{mult}_{p}H_{\phi_{r, d, N}(q)}=\textup{mult}_{p}H_{q}+\textup{mult}_{p}H^{N}\leq \textup{mult}_{p}H_{q}+(r-1)N.
\end{displaymath}\begin{displaymath}
\leq d+(r-1)N<rN <m+rN.
\end{displaymath}
Therefore, $[1:0:\ldots:0]$ is the unique point of $H_{\phi_{r, d, N}(q)}$ attaining maximal multiplicity. By \cite[Lemma 4.2]{hesselsing}, there is $g\in G$ such that $\Delta_{g.\phi_{r, d, N}(q)}$ is a worst state polytope of $\phi_{r, d, N}(q)$ and $[1:0:\ldots:0]$ is still the point of  $H_{g.\phi_{r, d, N}(q)}$ attaining maximal multiplicity. By the uniqueness of the point attaining maximal multiplicity, The dual action of $g$ on $\mathbb{P}_{k}^{r}$ fixes the point $[1:0:\ldots :0]$. That is, $g$ is in the form
\begin{displaymath}
g=\left[
\begin{array}{c|c}
1 & \mathbf{0} \\
\hline
\star & g'
\end{array}
\right]
\end{displaymath}
for some $g'\in \textup{SL}_{r}(k)$. Without loss of generality, $z_{i}\geq z_{i+1}$ for all $1\leq i\leq r-1$ where
\begin{displaymath}
\frac{\delta}{\Vert\lambda\Vert}\iota(h(\lambda_{g.\phi_{r, d, N}(q)})\otimes_{\mathbb{Z}}1)=(z_{0}, z_{1}, \ldots, z_{r}).
\end{displaymath}
There are upper triangular matrix $u$ with $1$'s in the diagonal, lower triangular matrix $l$ and a permutation matrix $\sigma$ satisfying $g'=ul\sigma$. Let $\overline{u}$ be the matrix of the form
\begin{displaymath}
\left[
\begin{array}{c|c}
1 & \mathbf{0} \\
\hline
\mathbf{0} & u
\end{array}
\right].
\end{displaymath}
Then, $\overline{u}^{-1}\in \pi(\lambda_{g.\phi_{r, d, N}(q)})$ so that $\Delta_{\overline{u}^{-1}g.\phi_{r, d, N}(q)}$ is a worst state polytope of $\phi_{r, d, N}(q)$ by Theorem~\ref{Kempftheorem}. Let $\overline{\sigma}$ be the matrix of the form
\begin{displaymath}
\left[
\begin{array}{c|c}
\det \sigma & \mathbf{0} \\
\hline
\mathbf{0} & \sigma
\end{array}
\right]\in\textup{SL}_{r+1}(k)
\end{displaymath}
 and let $<_{\textup{lex}}$ be the lexicographic monomial order satisfying
\begin{displaymath}
x_{0}>x_{1}>\ldots>x_{r}.
\end{displaymath}
Under the identification $\textup{Hilb}^{P_{r, d+rN}}(\mathbb{P}_{k}^{r})\cong\mathbb{P}(k[x_{0}, \ldots, x_{r}]_{d+rN})$, $\overline{\sigma}.\phi_{r, d, N}(q)=[f]$ for some $f\in k[x_{0}, \ldots, x_{r}]_{d+rN}$. Let $\eta$ be the leading monomial of $f$ with respect to $<_{\textup{lex}}$. We can write $f$ as follows:
\begin{displaymath}
f=c_{\eta}\eta+\sum_{n\in M_{d+rn}^{r}, n<_{\textup{lex}}\eta}c_{n}n
\end{displaymath}
where $M_{d}^{r}$ is the set of monomials in $k[x_{0}, \ldots, x_{r}]_{d}$ and $c_{n}\in k$ for all $n\in M_{d+rN}^{r}$. Note that $\deg_{x_{0}} \eta=d-m$ and the lattice point in $\mathbb{R}^{r+1}$ corresponding to $\eta$ is in $Q^{r}_{d, N}$. Let $\overline{l}=\overline{u}^{-1}g\overline{\sigma}^{-1}\in\textup{SL}_{r+1}(k)$ , which is in the form
\begin{displaymath}
\left[
\begin{array}{c|c}
\det{\sigma} & \mathbf{0} \\
\hline
\star & l
\end{array}
\right].
\end{displaymath}
We know that the leading monomial of $\overline{l}.f$ is still $\eta$, because $\overline{l}$ is lower-triangular. It means that the Newton polytope of $\overline{l}.f$ contains the lattice point corresponding to $\eta$. By the definition, $\Delta_{\overline{u}^{-1}g.\phi_{r, d, N}(q)}$ contains the lattice point corresponding to $\eta$. Thus,
\begin{displaymath}
\vert\Delta_{\overline{u}^{-1}g.\phi_{r, d, N}(q)}\vert\leq l_{d, N, m}^{r}
\end{displaymath}
by \eqref{maxlength}. $[1:0:\ldots:0]$ is still the point of $H_{\overline{u}^{-1}g.\phi_{r, d, N}(q)}$ which attains the maximal multiplicity so that
\begin{displaymath}
\frac{\delta}{\Vert\lambda\Vert}\iota(h(\lambda_{\overline{u}^{-1}g.\phi_{r, d, N}(q)})\otimes_{\mathbb{Z}}1) \in B_{d, N, m}^{r}
\end{displaymath}
by \cite[Lemma 4.1]{hesselsing}. By Theorem~\ref{polytope}, we know that
\begin{displaymath}
\phi_{r, d, N}(q) \in E_{[\lambda_{\overline{u}^{-1}g.\phi_{r, d, N}(q)}], \delta}^{d+Nr, r}
\end{displaymath}
and it completes the proof.
\end{proof}
When $N\gg 0$, our spheres centered at $\xi_{r, d+rN}$ looks like a plane around polynomial $Q_{d, N}^{r}$ so that we can say that the sets in $\{B_{d, N, m}^{r}|0\leq m\leq d\}$ are disjoint.
\begin{lemma}
\label{separation}
Fix $r$ and $d$. When $N> \frac{r-1}{2r}d^{2}+d$, $B_{d, N, m}^{r}\cap B_{d, N, m'}^{r}=\emptyset$ for all $0\leq m<m'\leq d$.
\end{lemma}
\begin{proof}
It suffices to show that
\begin{displaymath}
l_{d, N, m}^{r}<\min\bigg\{\vert\xi_{r, d+rN}-y\vert\bigg\vert y=(y_{i})_{i=0}^{r}\in \mathbb{R}_{\geq 0}^{r+1}, \sum_{i=0}^{r}y_{i}=d+rN, y_{0}=d-m' \bigg\}.
\end{displaymath}
if $N> \frac{r-1}{2r}d^{2}+d$. The right-hand side of the above inequality is equal to $\vert\xi_{r, d+rN}-y\vert$ when
\begin{displaymath}
y=z_{N}=\left( d-m',N+\frac{m'}{r}, \ldots, N+\frac{m'}{r} \right)
\end{displaymath}
by the convexity of the square-sum function. By definition,
\begin{displaymath}
\vert z_{N}-\xi_{r, d+rN}\vert^{2}-(l_{d, N, m}^{r})^{2}=
\end{displaymath}
\begin{displaymath}
\left(\frac{rd}{r+1}-\frac{rN}{r+1}-m'\right)^{2}+r\left(-\frac{d}{r+1}+\frac{N}{r+1}+\frac{m'}{r}\right)^{2}
\end{displaymath}
\begin{displaymath}
-\left(\frac{rd}{r+1}-\frac{rN}{r+1}-m\right)^{2}-\left(m-\frac{d}{r+1}+\frac{N}{r+1}\right)^{2}-(r-1)\left(-\frac{d}{r+1}+\frac{N}{r+1}\right)^{2}
\end{displaymath}
\begin{displaymath}
=\left(m'-m\right)\left(2N+2m'+2m-2d-\frac{(r-1)(m')^{2}}{r(m'-m)}\right).
\end{displaymath}
We know that
\begin{displaymath}
\frac{r-1}{2r}d^{2}+d\geq d+\frac{(r-1)(m')^{2}}{2r(m'-m)}-m-m'
\end{displaymath}
and it completes the proof.
\end{proof}
\section{Recovering constructible subsets of a Hilbert scheme indicating possible multiplicities from a Hesselink stratification}
By Lemma~\ref{restriction} and Lemma~\ref{separation}, we can separate two hypersurfaces whose multiplicities at $[1:0:\ldots:0]$ are different by using two Hesselink strata containing each destabilization. Let us define
\begin{displaymath}
F^{N}_{r, d, m}=\bigcup\bigg\{E_{[\lambda], \delta}^{d+Nr, r}\bigg|\lambda\in \Gamma(\textup{SL}_{r+1}(k)),\quad \delta \in\mathbb{R}_{>0},\quad \frac{\delta}{\Vert\lambda\Vert}\iota(h(\lambda)\otimes_{\mathbb{Z}}1)\in B_{d, N, m}^{r}\bigg\}.
\end{displaymath}
Then, we can prove the following theorem.
\begin{theorem}
\label{main}
For arbitrary $r$ and $d$ in $\mathbb{N}$,
\begin{displaymath}
\phi_{r, d, N}^{-1}\left( F_{r, d, m}^{N} \right)=\big\{q\in\textup{Hilb}^{P_{r, d}}(\mathbb{P}^{r}_{k})\big\vert\textup{mult}_{[1:0:\ldots:0]}H_{q}=m\big\}.
\end{displaymath}
for all $m\in\{0, 1, \ldots, d\}$ and $N>\frac{r-1}{2r}d^{2}+d$.
\end{theorem}
\begin{proof}
"$\subset$" is clear from Lemma~\ref{separation}. "$\supset$" is clear from Lemma~\ref{restriction}.
\end{proof}

 In the next example, we can see that Theorem~\ref{main} is sometimes very useful when we need to compute the Hesselink stratum containing a Hilbert point of very high degree.
\begin{example}
Let $f'=y^{3}+xz^{2}+z^{3}\in k[x, y, z]$. We know that $[f']\in \mathbb{P}(k[x, y, z]_{3})\cong\textup{Hilb}^{P_{2, 3}}(\mathbb{P}_{k}^{r})$ and $[y^{N}z^{N}f']\in \mathbb{P}(k[x, y, z]_{3+2N})\cong\textup{Hilb}^{P_{2, 3+2N}}(\mathbb{P}_{k}^{r})$ for all $N\in\mathbb{N}$. By Theorem~\ref{main}, $\phi_{2, 3, N}([f'])\in F_{2, 3, 2}^{N}$ for all $N\geq 6$ since $\textup{mult}_{[1:0:0]}H_{[f']}=2$. That is, $\phi_{2,3,N}([f'])\in E_{[\lambda_{N}], \delta_{N}}^{3+2N, 2}$ for all $N\geq 6$ where $\lambda_{N}\in\Gamma(\textup{SL}_{r+1}(k))$ and $\delta_{N}\in \mathbb{R}_{>0}$ satisfy
\begin{equation}
\label{lasteq}
\frac{\delta_{N}}{\Vert\lambda_{N}\Vert}\iota(h(\lambda_{N})\otimes_{\mathbb{Z}}1)\in B^{2}_{3, N, 2}
\end{equation}
for all $N\geq 6$. On the other hand, we know that $B_{3, N, 2}^{2}$ and $ \Delta_{\phi_{2,3,N}([f'])}$ meets at a unique point $(1, N, N+2)\in\mathbb{R}^{3}$ whenever $N\geq 6$ as we can see in figure~\ref{lastexample}. Therefore, by Theorem~\ref{polytope}, $ |\Delta_{\phi_{2,3,N}([f'])}|$ already attains the maximal value by \eqref{lasteq}, so that $\delta_{N}=|\Delta_{\phi_{2,3,N}([f'])}|=\sqrt{\frac{2}{3}N^{2}+2}$ and there is $g\in\textup{SL}_{3}(k)$ and $a\in\mathbb{N}$ such that $(g\star\lambda_{N})(t^{a})=\textup{diag} (t^{-2N}, t^{N-3}, t^{N+3})$ for all $t\in k^{\times}$, whenever $N\geq 6$.
\begin{figure}[h!]
\begin{tikzpicture}
\coordinate (up) at (0,sqrt 3);
\coordinate (left) at (-1.5,-1/2*sqrt 3);
\coordinate (right) at (1.5,-1/2*sqrt 3);
\filldraw [black] (barycentric cs:up=1,left=1,right=1) circle (1pt);
\filldraw [black] (barycentric cs:up=1,left=0,right=0) circle (1pt);
\node at (barycentric cs:up=1.1,left=-0.1,right=0) {$y^{9}z^{6}$};
\filldraw [black] (barycentric cs:up=2,left=1,right=0) circle (1pt);
\node at (barycentric cs:up=-0.4,left=1.4,right=2) {$xy^{6}z^{8}$};
\filldraw [black] (barycentric cs:up=2,left=0,right=1) circle (1pt);
\filldraw [black] (barycentric cs:up=1,left=2,right=0) circle (1pt);
\filldraw [black] (barycentric cs:up=1,left=0,right=2) circle (1pt);
\node at (barycentric cs:up=0,left=-0.2,right=1.2) {$y^{6}z^{9}$};
\filldraw [black] (barycentric cs:up=0,left=2,right=1) circle (1pt);
\filldraw [black] (barycentric cs:up=0,left=1,right=2) circle (1pt);
\filldraw [black] (barycentric cs:up=0,left=1,right=0) circle (1pt);
\filldraw [black] (barycentric cs:up=0,left=0,right=1) circle (1pt);
\filldraw [gray] (barycentric cs:up=1,left=0,right=0)--(barycentric cs:up=0,left=1,right=2)--(barycentric cs:up=0,left=0,right=1);
\filldraw [black] (barycentric cs:up=0,left=1,right=2) arc [start angle=atan(sqrt(3)/7), end angle=atan(sqrt(27)/5), radius=sqrt(13)];
\node at (barycentric cs:up=1.5,left=1.5,right=0) {$B_{3, 6, 2}^{2}$};
\node at (barycentric cs:up=1.5,left=-1.5,right=3) {$\Delta_{[y^{6}z^{6}(y^{3}+xz^{2}+z^{3})]}$};
\end{tikzpicture}
\caption{$B_{3, 6, 2}^{2}$ around the state polytope of $[y^{6}z^{6}(y^{3}+xz^{2}+z^{3})]$.}
\label{lastexample}
\end{figure}\\
\end{example}

 Theorem~\ref{main} means that a slice of the union of all strata in $F_{r, d, m}^{N}$ can be constructed by intersecting finitely many linear subspaces and  complements of linear subspaces if $N$ is large enough.

 Also, Theorem~\ref{main} means that we can recover the constructible subset
\begin{displaymath}
S_{r, d, m}=\big\{x\in\textup{Hilb}^{P_{r, d}}(\mathbb{P}_{k}^{r})\big\vert \textup{mult}_{p} H_{x}=m\textrm{ for some }p\in H_{x}\big\}
\end{displaymath}
of $\textup{Hilb}^{P_{r, d}}(\mathbb{P}_{k}^{r})$ from the Hesselink stratification of $\textup{Hilb}^{P_{r, d+rN}}(\mathbb{P}_{k}^{r})$ for sufficiently large $N\in\mathbb{N}$. That is,
\begin{corollary}
\label{last}
Suppose that $r$ and $d$ are in $\mathbb{N}$. Then,
\begin{displaymath}
G.\phi_{r, d, N}^{-1}\left( F_{r, d, m}^{N} \right)=S_{r, d, m}
\end{displaymath}
for all $m\in\{0, 1, \ldots, d\}$ and $N>\frac{r-1}{2r}d^{2}+d$.
\end{corollary}
\begin{example}
Let $f_{1}=xz(x^{2}+z^{2})\in k[x, y, z]_{4}$ and $f_{2}=y^{2}z^{2}\in k[x, y, z]_{4}$. Let us consider $[f_{1}], [f_{2}]\in \mathbb{P}(k[x, y, z]_{4})\cong \textup{Hilb}^{P}(\mathbb{P}^{r}_{k})$. Multiplicities of $H_{[f_{1}]}$ and $H_{[f_{2}]}$ at $[1:0:0]$ are $1$ and $4$, respectively. Thus, we know that $\phi_{2, 4, N}([f_{1}])\in F_{2, 4, 1}^{N}$ and $\phi_{2, 4, N}([f_{2}])\in F_{2, 4, 4}^{N}$ for all $N>8$ by Theorem~\ref{main}. On the other hand, we also know that maximal multiplicities of $H_{[f_{1}]}$ and $H_{[f_{2}]}$ are both 4. By \cite[Lemma 4.2]{hesselsing}, both $[f_{1}]$ and $[f_{2}]$ are in the stratum $E^{4, 2}_{[\lambda], 2\sqrt{2}/\sqrt{3} }$ where $\lambda(t)=\textup{diag} (t^{-2}, t, t)\in \textup{SL}_{3}(k)$ for all $t\in k^{\times}$. Furthermore, $H_{[f_{2}]}\notin S_{2, 4, 1}$ so that $\phi_{2, 4, N}(g.[f_{2}])\notin F_{2, 4, 1}^{N}$ for all $g\in\textup{SL}_{r+1}$ and $N>8$.
\end{example}

\bibliographystyle{plain}
\bibliography{lib}

\begin{thebibliography}{1}

\bibitem{Hesselink}
Wim~H. Hesselink.
\newblock Uniform instability in reductive groups.
\newblock {\em J. Reine Angew. Math.}, 303/304:74--96, 1978.

\bibitem{Kempf}
George~R. Kempf.
\newblock Instability in invariant theory.
\newblock {\em Ann. of Math. (2)}, 108(2):299--316, 1978.

\bibitem{hesselsing}
Cheolgyu Lee.
\newblock Instability and singularity of projective hypersurfaces.
\newblock {\em Proc. Amer. Math. Soc.}, 146(12):5015--5023, 2018.

\bibitem{instability}
Cheolgyu Lee.
\newblock Worst unstable points of a {H}ilbert scheme.
\newblock {\em J. Algebra}, 544:92--124, 2020.

\bibitem{Ian}
Ian Morrison and David Swinarski.
\newblock Gr\"obner techniques for low-degree {H}ilbert stability.
\newblock {\em Exp. Math.}, 20(1):34--56, 2011.

\bibitem{Mukai}
Shigeru Mukai.
\newblock {\em An introduction to invariants and moduli}, volume~81 of {\em
  Cambridge Studies in Advanced Mathematics}.
\newblock Cambridge University Press, Cambridge, 2003.
\newblock Translated from the 1998 and 2000 Japanese editions by W. M. Oxbury.

\bibitem{GIT}
D.~Mumford, J.~Fogarty, and F.~Kirwan.
\newblock {\em Geometric invariant theory}, volume~34 of {\em Ergebnisse der
  Mathematik und ihrer Grenzgebiete (2) [Results in Mathematics and Related
  Areas (2)]}.
\newblock Springer-Verlag, Berlin, third edition, 1994.

\bibitem{Nitin}
Nitin Nitsure.
\newblock Construction of {H}ilbert and {Q}uot schemes.
\newblock {\em ArXiv e-prints}, 2015.

\end{thebibliography}
\end{document}